\documentclass[12pt]{amsart}    
\usepackage[colorlinks]{hyperref}
\usepackage{amsmath,amssymb,latexsym, amscd}
\usepackage{mathabx}
\usepackage{exscale, cite, eps fig, graphics}
\usepackage{mathtools}
\usepackage{enumitem}
\usepackage{lipsum}
\usepackage{cancel}

\calclayout

\renewcommand{\geq}{\geqslant}
\renewcommand{\leq}{\leqslant}

\newcommand{\T}{\mathbb{T}}

\renewcommand{\v}{\boldsymbol{v}}

\newcommand{\g}{\gamma}

\newcommand{\R}{\mathbb{R}}
\newcommand{\Z}{\mathbb{Z}}
\newcommand{\C}{\mathbb{C}}

\newtheorem{theorem}{Theorem}[section]
\newtheorem{lemma}[theorem]{Lemma}

\newtheorem{definition}[theorem]{Definition} 
\newtheorem{corollary}[theorem]{Corollary}

\numberwithin{equation}{section}

\title[Transverse spectral instability in KD]{Transverse Spectral Instabilities in Konopelchenko-Dubrovsky Equation}
\author[Bhavna]{Bhavna$^\ast$}
\author[Pandey]{Ashish~Kumar~Pandey$^\ast$}
\author[Singh]{Sudhir~Singh$^\dagger$}
\address{$^\ast$Department of Mathematics, IIIT Delhi,  110020, India\\
$^\dagger$Department of Mathematics, NIT Trichy, 620015, India}
\email{bhavnai@iiitd.ac.in, ashish.pandey@iiitd.ac.in, sudhirew@gmail.com}
\date{\today}
\begin{document}
\maketitle
 \begin{abstract}
 We study the transverse spectral stability of the one-dimensional small-amplitude periodic traveling wave solutions of the  (2+1)-dimensional Konopelchenko-Dubrovsky (KD) equation. We show that these waves are transversely unstable with respect to two-dimensional perturbations that are periodic in both directions with long wavelength in the transverse direction. We also show that these waves are transversely stable with respect to perturbations which are either mean-zero periodic or square-integrable in the direction of the propagation of the wave and periodic in the transverse direction with finite or short wavelength. We discuss the implications of these results for special cases of the KD equation - namely, KP-II and mKP-II equations.
 \end{abstract}
 \section{Introduction}
 We consider the (2+1)-dimensional Konopelchenko-Dubrovsky (KD) equation  \cite{Konopelchenko1984SomeDimensions,Konopelchenko1992IntroductionEquations} 
\begin{equation}\label{e:KD}
\centering
\left\{\begin{aligned}
u_t-u_{xxx}-6 \rho u u_x +\dfrac{3}{2}\phi^2 u^2 u_x-3 v_{y}+3 \phi u_x v &= 0, \\
u_y &= v_x, \\ 
\end{aligned}\right.
 \end{equation} 
where, $u=u(x,y,t)$, $v=v(x,y,t)$, the subscripts denote partial differentiation, $\rho$ and $\phi$ are real parameters, defining the magnitude of nonlinearity in wave propagation, modelled for stratified shear flow, the internal and
shallow-water waves and the plasmas \cite{Xu2011PainleveComputation}, can also be regarded as combined KP
and modified KP equation \cite{Zhang2009Symbolic-computationMethod}, or generalized (2+1)D Gardner equation \cite{Konopelchenko1991InverseEquation}.

\subsection*{Models} The $(1+1)$-dimensional reduction of the KD equation \eqref{e:KD} is the Gardner equation\cite{Krishnan2011AEquation,Xu2009AnalyticPhysics}
\begin{equation}\label{e:gar}
u_t-u_{xxx}-6\rho u u_x+\dfrac{3}{2}\phi^2u^2u_x=0,
\end{equation}which is example of the generalized Korteweg-de-Vries equation(gKdV)\cite{Korteweg1895Waves}, that is
\[u_t+(g(u)-u_{xx})_x=0,\]
where $g:\mathbb{R}\to\mathbb{R}$ is a smooth real function. Gardner equation \eqref{e:gar} can be reduced to KdV and modified KdV equations for $\phi=0$ and $\rho=0$, respectively. 

For $\phi=0$, \eqref{e:KD} is Kadomtsev-Petviashvili (KP) equation with negtaive dispersion \cite{Kadomtsev1970OnMedia}
\begin{equation}\label{e:KP}
    (u_t-u_{xxx}-6\rho u u_x)_x-3u_{yy}=0,
\end{equation}
which is also known as KP-II equation. Modified KP-II\cite{Sun2009InelasticElectrodynamics}(say, mKP-II) reads from \eqref{e:KD} for $\rho=0$,
\begin{equation}\label{e:mKP}
\left(u_t-u_{xxx}+\dfrac{3}{2}\phi^2u^2u_x\right)_x-3u_{yy}=0.
\end{equation}
\subsection*{Integrability} The KD equation \eqref{e:KD} is integrable \cite{Konopelchenko1984SomeDimensions,  Zhang2009Symbolic-computationMethod, Maccari1999AEquation}. Integrability is an useful property to have for evolution equations, especially in higher dimensions. It gives sufficient freedom to explore the equation through different aspects. It also helps significantly to observe nonlinear coherent structures like rogue waves, breathers, solitons and elliptic waves in the systems \cite{Ma2020MultipleEquation, Liu2019LumpEquation,  Wu2018ComplexitonEquation, Yuan2018SolitonsEquations, Ren2016TheSolutions}. Some well-known integrable water wave models are classical KdV, KP, and Schr\"odinger equations. The KD equation \eqref{e:KD}, like similar evolution equations in (2+1) dimension, for example, Kadomtsev-Petviashvili equation, Davey-Stewarson equation, and the three-wave equations, is solvable through  Inverse Scattering Transform (IST) \cite{Konopelchenko1984SomeDimensions}.  It is among the few nonlinear evolution equations which are completely integrable in different settings. Notably, the considered KD equation is also integrable in the  Painlev\'e sense and solvable through IST \cite{Xu2011PainleveComputation, Konopelchenko1984SomeDimensions,  Konopelchenko1991InverseEquation}.
\subsection*{Dispersion Relation} Assuming a plane-wave solution of the form 
\[u(x,y,t)=e^{i(kx-\Omega t+\gamma y)},
\]
for the linear part 
\[(u_t-u_{xxx})_x-3u_{yy}=0,
\]
of the KD equation \eqref{e:KD}, we arrive at the dispersion relation 
\[\Omega(k)=k^3-\dfrac{3\gamma^2}{k}.
\]
\subsection*{Small amplitude periodic traveling waves} The $y$-independent periodic traveling wave solution of the KD equation \eqref{e:KD} that are also solutions of the Gardner equation \eqref{e:gar}, is of the form \[\begin{pmatrix}
u(x,y,t)\\v(x,y,t)
\end{pmatrix}=\begin{pmatrix}
u(x-ct)\\v(x-ct)
\end{pmatrix},
\]for some $c\in\mathbb{R}$. Under this assumption, we arrive at
\begin{equation}\label{e:pt}
\centering
\left\{\begin{aligned}
    -cu_x-u_{xxx}+\dfrac{\phi^2}{2}(u^3)_x-3\rho (u^2)_x+3\phi u_xv=0,\\v_x=0,\\
\end{aligned}\right.
\end{equation} which implies $v=b_1$, where $b_1$ is an arbitrary constant. Substituting $v=b_1$ and integrating, \eqref{e:pt} is reduced to
\begin{equation}\label{e:pt1}
   -cu-u_{xx}+\dfrac{\phi^2}{2}u^3-3\rho u^2+3\phi u b_1=b_2,
\end{equation} where $b_1,b_2\in\mathbb{R}$. Let $u$ be a $2\pi/k$-periodic function of its argument, for some $k>0$. Then, $w(z):=u(x)$ with $z=kx$, is a $2\pi$-periodic function in $z$, satisfying
\begin{equation}
    -cw-k^2w_{zz}+\dfrac{\phi^2}{2}w^3-3\rho w^2+3\phi wb_1=b_2,
\end{equation}
For a fixed $\phi$ and $\rho$, let $F:H^2(\mathbb{T})\times \mathbb{R}\times \mathbb{R}^+\times\mathbb{R}\times\mathbb{R}
\to L^2(\mathbb{T)}$ be defined as 
\[
F(w,c;k,b_1,b_2)=-cw-k^2w_{zz}+\dfrac{\phi^2}{2}w^3-3\rho w^2+3\phi wb_1-b_2.
\]
We try to find a solution $w\in H^2(\mathbb{T})$ of
\begin{equation}\label{e:f}
   F(w,c;k,b_1,b_2)=0.
\end{equation}
For any $c\in\mathbb{R}$, $k>0$, $b_1,b_2\in\mathbb{R}$ and $|b_1|, |b_2|$ sufficiently small, note that
\begin{equation}
  w_0(c,k,b_1,b_2)=- \dfrac{1}{c}b_2+O((b_1+b_2)^2),
\end{equation}
make a constant solution of \eqref{e:f}. Note that $w_0\equiv 0$ if $b_1=b_2=0$. If non-constant solutions of \eqref{e:f} bifurcate from $w_0\equiv 0$ for some $c=c_0$ then $\text{ker}(\partial_wF(0,c_0;k,0,0))$ is non-trivial. Note that
\[\text{ker}(\partial_wF(0,c_0;k,0,0))=\text{ker}(-c_0-k^2\partial_z^2)=\text{span}\{e^{\pm iz}\},
\]provided that $c_0=k^2$. 

The periodic traveling waves of \eqref{e:gar} exist (see, \cite{Bronski2016ModulationalType}), and by following the Lyapunov-Schmidt procedure, their small-amplitude expansion is as follows.
\begin{theorem}
For any $k>0$, $b_1,b_2\in\mathbb{R}$ and $|b_1|,|b_2|$ sufficiently small, a one parameter family of solutions of \eqref{e:KD}, denoted by\[
\begin{pmatrix}
u(x,t)\\v(x,t)
\end{pmatrix}=\begin{pmatrix}w(a,b_1,b_2)(z)\\v(z)
\end{pmatrix}\]where $z=k(x-c(a,b_1,b_2)t)$, $|a|$ sufficiently small, $w(a,b_1,b_2)(z)$ is smooth, even and $2\pi$-periodic in $z$ and $c$ is even in $a$, is given by
\begin{equation}\label{e:expptw}
\left\{
\begin{aligned}
w(a,b_1,b_2)(z)=&-\dfrac{1}{k^2}b_2+a\cos z+a^2(A_0+A_2\cos 2z)+a^3A_3\cos 3z+O(a^4+a^2(b_1+b_2)^2),\\
v(z)=&b_1,\\
c(a,b_1,b_2)=&k^2+3\phi b_1+\dfrac{3\rho}{k^2}b_2+a^2c_2+O(a^4+a^2(b_1+b_2)^2),
\end{aligned}\right.
\end{equation}
where
\begin{equation}
A_0=-\dfrac{3\rho}{2k^2},\quad A_2=\dfrac{\rho}{2k^2},\quad A_3=-\dfrac{\phi^2}{64k^2}+\dfrac{3\rho^2}{16k^4},\quad \text{and}\quad c_2=\dfrac{3\phi^2}{8}+\dfrac{15\rho^2}{2k^2}.
\end{equation}
\end{theorem}
\subsection*{Transverse stability} For the KD equation \eqref{e:KD}, the solution and integrability aspects have been studied thoroughly, see \cite{Xu2010IntegrableEquation, Zhang2009Symbolic-computationMethod, Yang2008TravellingMethod,    Kumar2016SimilarityTheory}, for instance. However, to the best of the authors' knowledge, any result on the stability aspects of the periodic traveling waves or solitary waves of the KD equation \eqref{e:KD} has not been discussed so far. However, the transverse instability of periodic traveling waves has been studied for many similar equations, for instance, for the KP equation in \cite{Bhavna2021TransverseEquation,Hakkaev2012TransverseEquations,Johnson2010TransverseEquation,Haragus2011TransverseEquation}, for Zakharov-Kuznetsov (ZK) equation in \cite{Chen2012AEquation,Johnson2010TheEquations}. Transverse instability of solitary wave solutions of various water-wave models has also been explored by several authors, see \cite{Groves2001TransverseWaves,Pego2004OnTension,Rousset2009TransverseModels,Rousset2011TransverseWater-waves}. Motivated by the importance of nonlinear waves propagation and its stability, we investigate the transverse spectral instability of the KD equation. We aim to study the (in)stability of the $y$-independent, that is, (1+1)-dimensional periodic traveling waves \eqref{e:expptw} of \eqref{e:KD} with respect to two-dimensional perturbations which are either periodic or non-periodic in the $x$-direction and always periodic in the $y$-direction. The periodic nature of the perturbations in the $y$-direction is classified into two categories: long wavelength and finite or short-wavelength transverse perturbations. The (in)stabilities that occur due to long-wavelength transverse perturbations are termed as {\em modulational transverse (in)stabilities}. Furthermore, we use the term {\em high-frequency transverse (in)stabilities} to refer to those transverse (in)stabilities that are occurring due to finite or short-wavelength transverse perturbations. 
Moreover, depending on the periodic or non-periodic nature of perturbations in the direction of propagation of the one-dimensional wave, we term the resulting instability as transverse instability with respect to periodic or non-periodic perturbations, respectively. Our main results are the following theorems depicting the transverse stability and instability of small amplitude periodic traveling waves \eqref{e:expptw} of \eqref{e:KD}.


\begin{theorem}[Transverse stability]\label{t:3}
Assume that small-amplitude periodic traveling waves \eqref{e:expptw} of \eqref{e:KD} are spectrally stable in $L^2(\mathbb T)$ as a solution of the corresponding $y$-independent one-dimensional equation. Then, for any $a$ sufficiently small, $\rho\in \mathbb{R}$, $\phi\in \mathbb{R}$, and $k>0$, periodic traveling waves \eqref{e:expptw} of \eqref{e:KD} are transversely stable with respect to two-dimensional perturbations which are either mean-zero periodic or non-periodic (localized or bounded) in the direction of propagation and finite or short wavelength in the transverse direction.
\end{theorem}


\begin{theorem}[Transverse instability]\label{t:2} For a fixed $\rho\in\mathbb{R}$ and $\phi\neq 0$, sufficiently small amplitude periodic traveling waves \eqref{e:expptw} of KD equation suffers modulational transverse instabilities with respect to periodic perturbations if 
\[
k>2\left|\dfrac{\rho}{\phi}\right| \quad \text{and} \quad |\gamma|<k|a|\sqrt{\left|\dfrac{\phi^2}{4}-\dfrac{\rho^{2}}{k^2}\right|}+O(a(\gamma+a)).
\]
\end{theorem}


As a consequence of these theorems, for all $k>0$, periodic traveling waves \eqref{e:expptw} of the mKP-II equation suffers modulational transverse instability with respect to periodic perturbations, which is in accordance with results in  \cite{Johnson2010TransverseEquation}.
Also, in the limit $\phi\to 0$, there is no modulational transverse instability for KP-II equation by Theorem~\ref{t:2} which again agrees with results in \cite{Spektor1988StabilityDispersion,Haragus2011TransverseEquation,Johnson2010TransverseEquation,HLP17,Bhavna2021TransverseEquation}. From Theorem \ref{t:3}, KP-II does not possess any high-frequency transverse instability \cite{HLP17}. Moreover, we have the following stability result for the mKP-II equation using Theorem~\ref{t:3}.  

\begin{corollary}
For all $k>0$, sufficiently small
amplitude periodic traveling waves \eqref{e:expptw} of the mKP-II equation does not possess any high-frequency transverse instability with respect to both mean-zero periodic and non-periodic perturbations.
\end{corollary}


In Section~\ref{sec:lin}, we linearize the equation and formulate the problem. In Section~\ref{s:3}, we list all potentially unstable nodes. In Sections \ref{s:4} and \ref{s:5}, we provide transverse instability analysis to investigate modulational and high-frequency transverse instabilities with respect to periodic and non-periodic perturbations. 

\subsection*{Notations}\label{sec:notations}
Throughout the article, we have used the following notations. 
Here, $L^{2}(\mathbb{R})$ is the set of Lebesgue measurable, real, or complex-valued functions over $\mathbb{R}$ such that
$$
\|f\|_{L^{2}(\mathbb{R})}=\left(\int_{\mathbb{R}}|f(x)|^{2} d x\right)^{1 / 2}<+\infty,
$$
and, $L^{2}(\mathbb{T})$ denote the space of $2 \pi$-periodic, measurable, real or complex-valued functions over $\mathbb{R}$ such that
$$
\|f\|_{L^{2}(\mathbb{T})}=\left(\frac{1}{2 \pi} \int_{0}^{2 \pi}|f(x)|^{2} d x\right)^{1 / 2}<+\infty.
 $$
 Here, $L^2_0(\mathbb{T})$ is the space of square-integrable functions with of zero-mean, 
\begin{equation}\label{e:zerom}
    L^2_0(\T)=\left\{f\in L^2(\T)\;:\;\int_0^{2\pi}f(z)~dz=0\right\}.
\end{equation}
The space $C_{b}(\mathbb{R})$ consists of all bounded continuous functions on $\mathbb{R}$, normed with
$$
\|f\|=\sup _{x \in \mathbb{R}}|f(x)|.
$$
For $s \in \mathbb{R}$, let $H^{s}(\mathbb{R})$ consists of tempered distributions such that
$$
\|f\|_{H^{s}(\mathbb{R})}=\left(\int_{\mathbb{R}}\left(1+|t|^{2}\right)^{s}|\hat{f}(t)|^{2} d t\right)^{\frac{1}{2}}<+\infty,
$$
and
$$
H^{s}(\mathbb{T})=\left\{f \in H^{s}(\mathbb{R}): f \text { is } 2 \pi \text {-periodic }\right\}.
$$
We define $L^{2}(\mathbb{T})$-inner product as
$$
\langle f, g\rangle=\frac{1}{2 \pi} \int_{0}^{2 \pi} f(z) \bar{g}(z) d z=\sum_{n \in \mathbf{Z}} \hat{f}_{n} \overline{\hat{g}_n},
$$
where $\widehat{f}_{n}$ are Fourier coefficients of the function $f$ defined by
$$
\widehat{f}_{n}=\frac{1}{2 \pi} \int_{0}^{2 \pi} f(z) e^{i n z} d z.
$$
Throughout the article, $\Re(\mu)$ represents the real part of $\mu\in\mathbb{C}$.
\section{Linearization and the spectral problem set up}\label{sec:lin}
Linearizing \eqref{e:KD} about its one-dimensional periodic traveling wave solution $\begin{pmatrix}w\\v\end{pmatrix}$ given in \eqref{e:expptw}, and considering the perturbations to $\begin{pmatrix}w\\v\end{pmatrix}$ of the form
\begin{equation}
    \begin{pmatrix}w\\v\end{pmatrix}+\epsilon\begin{pmatrix}\zeta\\\psi\end{pmatrix}+O(\epsilon^2)\quad\text{for}\quad 0<|\epsilon|\ll 1
\end{equation}
we arrive at, 
\begin{equation}\label{e:lin1}
\centering
\left\{\begin{aligned}
    \zeta_t-kc\zeta_z-k^3\zeta_{zzz}-6k\rho(w\zeta)_z+\dfrac{3}{2}\phi^2k(w^2\zeta)_z-3\psi_y+3\phi k w_z\psi+3\phi kb_1\zeta_z&=0,\\
  \zeta_y-k\psi_z&=0.
    \end{aligned}\right.
\end{equation}
We seek a solution of the form $\begin{pmatrix}\zeta(z,t,y)\\ \psi(z,t,y) \end{pmatrix}=e^{\mu t+i\gamma y} \begin{pmatrix}\zeta(z)\\ \psi(z) \end{pmatrix}$, $\mu\in\mathbb{C}$, $\gamma\in\mathbb{R}$, of \eqref{e:lin1}, which leads to
\begin{equation}\label{e:lin2}
\centering
\left\{\begin{aligned}
    \mu\zeta-kc\zeta_z-k^3\zeta_{zzz}-6k\rho(w\zeta)_z+\dfrac{3}{2}\phi^2k(w^2\zeta)_z-3i\gamma\psi+3\phi k w_z\psi+3\phi kb_1\zeta_z&=0,\\
    i\gamma\zeta-k\psi_z&=0.
    \end{aligned}\right.
\end{equation}
We can reduce this system of equations into 
\begin{equation}
\begin{aligned}\label{e:opt}
   &\mathcal Q_{a,b_1,b_2}(\mu,\gamma)\psi:=\\& \left(k\left(\mu-kc\partial_z-k^3\partial_z^3-6k\rho\partial_z(w\cdot)+\dfrac{3}{2}\phi^2k\partial_z(w^2\cdot)\right)\partial_z+3\gamma^2+3\phi k(i\gamma w_z+kb_1\partial_z^2)\right)\psi=0.
\end{aligned}\end{equation}
\begin{definition}(Transverse (in)stability)
Assuming that $2\pi/k$-periodic traveling wave solution $\begin{pmatrix}u(x,y,t\\v(x,y,t\end{pmatrix}=\begin{pmatrix}w(k(x-ct))\\v(k(x-ct))\end{pmatrix}$ of \eqref{e:KD} is a stable solution of the one-dimensional equation \eqref{e:gar} where $w$, $v$ and $c$ are as in \eqref{e:expptw}, we say that the periodic wave $\begin{pmatrix}w\\v\end{pmatrix}$ in \eqref{e:expptw} is transversely spectrally stable with respect to two-dimensional periodic perturbations (resp. non-periodic (localized or bounded perturbations)) if the KD operator $\mathcal Q_{a,b_1,b_2}(\mu,\gamma)$ acting in $L^2(\T)$ (resp.  $L^2(\R)$ or $C_b(\R)$) is invertible, for any $\mu\in\C$, $\Re(\mu)>0$ and any $\gamma\neq 0$ otherwise it is deemed transversely spectrally unstable. 
\end{definition}

We split the study of invertibility of $\mathcal Q_{a,b_1,b_2}(\mu,\gamma)$ into periodic ($L^2(\T)$) and non-periodic perturbations ($L^2(\R)$ or $C_b(\R)$). In further study, we assume $b_1=b_2=0$. For nonzero $b_1$ and $b_2$, one may explore in like manner. However, the
calculation becomes lengthy and tedious.
\subsection*{Periodic perturbations}\label{s:per}
Here, we are considering perturbations which are periodic in $z$, that is, in the direction of the propagation of wave. We check  the invertibility of the operator $\mathcal Q_{a,b_1,b_2}(\mu,\gamma)$ acting in $L^2(\T)$ for any $\mu\in\C$, $\Re(\mu)>0$ and any $\gamma\neq 0$. We use the notation $\mathcal Q_{a}(\mu,\gamma)$ for $\mathcal Q_{a,b_1,b_2}(\mu,\gamma)$ for simplicity. We convert the invertibility problem
\[\mathcal Q_{a}(\mu,\gamma)\psi=0;\quad \psi\in L^2(\T)\] 
into a spectral problem which requires invertibility of $\partial_z$. Since $\partial_z$ is not invertible in $L^2(\mathbb{T})$, we restrict the problem to mean-zero subspace $L^2_0(\mathbb{T})$, defined in \eqref{e:zerom}, of $L^2(\mathbb{T})$. Since $L_0^2(\T)\subset L^2(\T)$, if the operator $\mathcal{Q}_a(\mu,\gamma)$ is not invertible for some $\mu\in\C$ in $L^2_0(\T)$ implies that the operator $\mathcal{Q}_a(\mu,\gamma)$ is not invertible in $L^2(\T)$ as well for the same $\mu\in\C$.
The operator $\mathcal Q_{a}(\mu,\gamma)$ acting on $L^2_0(\mathbb{T})$ has a compact resolvent so that the spectrum consists of isolated eigenvalues with finite multiplicity. Therefore, $\mathcal Q_{a}(\mu,\gamma)$ is invertible in $L^2_0(\mathbb{T})$ if and only if zero is not an eigenvalue of $\mathcal Q_{a}(\mu,\gamma)$. Using this and the invertibility of $\partial_z$ in $L_0^2(\mathbb{T})$, we have the following result.

\begin{lemma}\label{lem:mh}The operator $\mathcal Q_{a}(\mu,\gamma)$ is not invertible in $L^2_0(\mathbb{T})$ for some $\mu\in \C$ if and only if  $\mu\in\operatorname{spec}_{L^2_0(\mathbb{T})}(\mathcal H_{a}(\gamma))$, that is, $L_0^2(\mathbb{T})$-spectrum of the operator, where
\begin{align*}
	\mathcal H_{a}(\gamma):= ck \partial_z+k^{3} \partial_{z}^{3}+6 k \rho \partial_z(w)-\frac{3}{2} \phi^{2} k \partial_z\left(w^{2}\right)-\frac{3 \gamma^{2}}{k} \partial_z^{-1} -i 3\phi \gamma w_{z} \partial_z^{-1}.
\end{align*}
\end{lemma}
\begin{proof}
The operator $\mathcal Q_{a}(\mu,\gamma)$ is not invertible in $L_0^2(\mathbb{T})$ for some $\mu\in \C$, if and only if zero is an eigenvalue of $\mathcal Q_{a}(\mu,\gamma)$. Moreover, for a $\varphi\in L_0^2(\mathbb{T})$, $\mathcal Q_{a}(\mu,\gamma)\varphi=0$ if and only if $\mathcal H_{a}(\gamma)\varphi=\mu \varphi$. The proof follows trivially.
\end{proof}

Next, we analyze the spectrum of the operator $\mathcal H_{a}(\gamma)$ acting in  $L^2_0(\T)$ with domain $H^{3}(\T)\cap L^2_0(\T)$. Since $w$ is an even function, $w_z$ is an odd function and therefore, the spectrum of $\mathcal{H}_a(\gamma)$ is not symmetric with respect to the reflection through the origin. Moreover, the operator $\mathcal{H}_a(\gamma)$ is not real, therefore the spectrum of $\mathcal{H}_a(\gamma)$ is not symmetric with respect to the real axis as well. The spectrum of $\mathcal{H}_a(\gamma)$ inherits following symmetry property.
\begin{lemma}\label{lem:sym}
The spectrum of $\mathcal{H}_a(\gamma)$ is symmetric with respect to the reflection through the imaginary axis.
\end{lemma}
\begin{proof}
We consider $\mathcal{R}$ to be the reflection through the imaginary axis defined as follows
\begin{equation*}
    \mathcal{R}\psi(z)=\overline{\psi(-z)}
\end{equation*}
Assume $\mu$ is the eigenvalue of $\mathcal{H}_a(\gamma)$ with an associated eigenvector $\varphi$, then we have
\begin{equation}\label{e:eig}
    \mathcal{H}_a(\gamma)\varphi=\mu\varphi
\end{equation}
Since
\[(\mathcal{H}_a(\gamma)\mathcal{R}\psi)(z)=\mathcal{H}_a(\gamma)(\mathcal{R}\psi(z))=\mathcal{H}_a(\gamma)\overline{\psi(-z)}=-(\overline{\mathcal{H}_a(\gamma)\psi})(-z)=-(\mathcal{R}\mathcal{H}_a(\gamma)\psi)(z),
\]
therefore, $\mathcal{H}_a(\gamma)$ anti-commutes with $\mathcal{R}$. Using \eqref{e:eig}, we arrive at
\[\mathcal{H}_a(\gamma)\mathcal{R}\varphi=-\mathcal{R}\mathcal{H}_a(\gamma)\varphi=-\overline{\mu}\mathcal{R}\varphi
\]
We conclude from here that if $\mu$ is an eigenvalue of $\mathcal{H}_a(\gamma)$ with associated eigenvector $\varphi$, then $-\overline{\mu}$ is also an eigenvalue of $\mathcal{H}_a(\gamma)$ with associated eigenvector $\mathcal{R}\varphi.$
\end{proof}
\subsection*{Non-periodic perturbations}\label{s:nper}
With respect to these perturbations, we aim to study the invertibility of $\mathcal Q_a(\mu, \gamma)$ acting in $L^2(\R)$ or $C_b(\R)$ (with domain $H^{4}(\R)$ or   $C_b^{4}(\R)$),
for $\mu\in\C$, $\Re(\mu)>0$, and $\gamma\in\R$, $\gamma\neq0$. In $L^2(\R)$ or $C_b(\R)$, the operator $\mathcal Q_a(\mu, \gamma)$ no longer have point isolated spectrum, rather it have continuous spectrum. Thus, we rely upon the Floquet Theory such that all solutions of \eqref{e:opt} in $L^2(\R)$ or $C_b(\R)$ are of the form $\psi(z)=e^{i\tau z}\Psi(z)$ where $\tau\in\left(-\frac12,\frac12\right]$ is the Floquet exponent and $\Psi(z)$ is a $2\pi$-periodic function, see \cite{Haragus2008STABILITYEQUATION} for a similar situation. By following same arguments as in the proof of \cite[Proposition A.1]{Haragus2008STABILITYEQUATION}, we can infer that the study of the invertibility of $\mathcal Q_a(\mu,\gamma)$ in $L^2(\R)$ or $C_b(\R)$ is equivalent to the invertibility of the linear operators $\mathcal Q_{a,\tau}(\mu,\gamma)$ in $L^2(\mathbb{T})$ with domain $H^{4}(\mathbb{T})$, for all $\tau\in\left(-\frac12,\frac12\right]$, where
\begin{align*}
\mathcal Q_{a,\tau}(\mu,\gamma) = \left(\mu-kc(\partial_z+i\tau)-k^3(\partial_z+i\tau)^3-6k\rho(\partial_z+i\tau)(w)\right)(k(\partial_z+i\tau))\\+\left(\dfrac{3}{2}\phi^2k(\partial_z+i\tau)(w^2)\right)(k(\partial_z+i\tau))+3\gamma^2+i3k\gamma\phi w_z.
\end{align*}

Since $\tau=0$ corresponds to the periodic perturbations we have already investigated, we would now restrict ourselves to the case of $\tau\neq0$. The $L^2(\mathbb{T})$-spectra of operator $\mathcal Q_{a,\tau}(\mu,\gamma)$ consist of eigenvalues of finite multiplicity. Therefore, $\mathcal Q_{a,\tau}(\mu,\gamma)$ is invertible in $L^2(\mathbb{T})$ if and only if zero is not an eigenvalue of $\mathcal Q_{a,\tau}(\mu,\gamma)$. We have the following result using this and the invertibility of $\partial_z+i\tau$.
\begin{lemma} The operator $\mathcal Q_{a,\tau}(\mu,\gamma)$ is not invertible in $L^2(\mathbb{T})$ for some $\mu\in \C$ and $\tau\neq 0$ if and only if  $\mu\in\phi(\mathcal H_a(\gamma,\tau)))$, $L^2(\mathbb{T})$-spectrum of the operator,
\begin{align*}
	\mathcal{H}_a(\gamma,\tau):=k c (\partial_z+i\tau)&+k^{3} (\partial_z+i\tau)^{3}+6 k \rho (\partial_z+i\tau)(w)\\&-\frac{3}{2} \phi^{2} k (\partial_z+i\tau)\left(w^{2}\right)-\left(\dfrac{3 \gamma^{2}}{k} +i 3\phi \gamma w_{z}\right) (\partial_z+i\tau)^{-1}.\\
\end{align*}
\end{lemma}
\begin{proof}
The proof is similar to Lemma~\ref{lem:mh}.
\end{proof}
 We will study the $L^2(\mathbb{T})$-spectra of linear operators $\mathcal H_a(\gamma,\xi)$ for $|a|$ sufficiently small, and for $|\tau|>\delta>0$ since the operator $(\partial_z+i\tau)^{-1}$ becomes singular, as $\tau\to 0$. Note that the spectrum of $\mathcal{H}_a(\g,\tau)$ is not symmetric with respect to the reflection through real axis or origin. Instead, we have the following symmetry
 \begin{lemma}
 The spectrum of $\mathcal{H}_a(\gamma,\tau)$ is symmetric with respect to the reflection through the imaginary axis for all $\tau\in\left(-\frac12,\frac12\right]\setminus\{0\}$.
 \end{lemma}
 \begin{proof}
 The proof is similar to Lemma \ref{lem:sym}.
 \end{proof}
\section{Characterization of the unperturbed spectrum}\label{s:3}
\subsection{Periodic perturbations}\label{ss:1}
As a consequence of the symmetry of the spectrum obtained in Lemma~\ref{lem:sym}, we obtain instability if there is an eigenvalue of $\mathcal{H}_a(\gamma)$ off the imaginary axis. A straightforward calculation reveals that 
\begin{align}\label{E:spec1}
    \mathcal H_{0}(\gamma)e^{inz} = i\Omega_{n,\gamma}e^{inz}\quad \text{for all}\quad n \in \mathbb{Z}^\ast:=\mathbb{Z}\setminus \{0\}.
\end{align}
where
\begin{align}\label{E:omega}
    \Omega_{n,\gamma} = k^{3} n\left(1-n^{2}\right)+\dfrac{3 \gamma^{2}}{kn}.
\end{align}
Therefore, the $L_0^2(\T)$-spectrum of $\mathcal H_0(\gamma)$ is given by
\begin{equation*}\label{e:spec2}
    \operatorname{spec}_{L^2_0(\mathbb{T})}(\mathcal H_0(\gamma))=\{i\Omega_{n,\gamma}; n \in \Z^\ast \},
\end{equation*}
which implies $\operatorname{spec}_{L^2_0(\mathbb{T})}(\mathcal H_0(\gamma))$ consists of purely imaginary eigenvalues of finite multiplicity. This is because the coefficients of the operator $\mathcal{H}_0(\gamma)$ are real, which should be the case, since zero-amplitude solutions are spectrally stable.

Spectra of $\mathcal H_a(\gamma)$ and $\mathcal H_0(\gamma)$ remain close for $|a|$ small as 
\[
||\mathcal H_a(\gamma)-\mathcal H_0(\gamma)||\to 0 \text{ as } a \to 0
\]
in the operator norm. Due to the symmetry in Lemma~\ref{lem:sym}, for $|a|$ sufficiently small, bifurcation of eigenvalues of $\mathcal H_a(\gamma)$ from imaginary axis can happen only when a pair of eigenvalues of $\mathcal H_0(\gamma)$ collide on the imaginary axis. 
Let $n\neq m\in \Z^\ast$, a pair of eigenvalues $i\Omega_{n,\gamma}$ and $i\Omega_{m,\gamma}$ of $\mathcal H_0(\gamma)$ collide for some $\gamma=\gamma_c$ when 
\begin{align}\label{e:coll}
    \Omega_{n,\gamma_c}=\Omega_{m,\gamma_c}.
\end{align}
We list all the collisions in the following lemma.
\begin{lemma}\label{lem111}
For a fix $\Delta \in \mathbb{N}$,  eigenvalues $\Omega_{n,\gamma}$ and $\Omega_{n+\Delta,\gamma}$ of the operator $\mathcal{H}_0(\gamma)$ collide for all $n\in(-\Delta,0)\cap \mathbb{Z}$ at some $\gamma=\gamma_c(k)$. All such collisions take place away from the origin in the complex plane except when $\Delta$ is even and $n=-\Delta/2$ in which case eigenvalues $\Omega_{n,\gamma}$ and $\Omega_{-n,\gamma}$ collide at the origin.
\end{lemma}
\begin{proof}
Without any loss of generality, consider $m>n$ and $m=n+\Delta$ with $\Delta\in \mathbb{N}$ in the collision condition \eqref{e:coll} then we obtain
\begin{equation}\label{e:cc}
    3\gamma_c^2=k^4n(n+\Delta)(-3n^2-3n\Delta-\Delta^2+1),
\end{equation}
which can be rewritten as
\begin{equation}\label{e:cc123}
    3\gamma_c^2=-k^4[3n^2(n+\Delta)^2+n(n+\Delta)(\Delta^2-1)].\end{equation}
The above equation implies that collision between $n$ and $n+\Delta$ takes place if only if $n(n+\Delta)<0$, that is, $-\Delta<n<0$. 
Observe that $\Omega_{n,\g_c}=\Omega_{-n,\g_c}=0$ for $\gamma_c^2=\dfrac{k^4n^2(n^2-1)}{3}$. Therefore, $\Omega_{n,\gamma_c}$ and $\Omega_{n+\Delta,\gamma_c}$ collide at the origin when $\Delta$ is even and $n=-\Delta/2$. All other collisions are away from origin. Hence, the lemma.
\end{proof}
From \eqref{e:cc123}, assume  $\gamma^2=-\dfrac{k^4}{3}f(n)g(n)$, where $f(n)=n(n+\Delta)$ and $g(n)=3n^2+3n\Delta+\Delta^2-1$. For a fixed $\Delta\in \mathbb{N}$, $f(n)\geq f(-\Delta/2)$ and $g(n)\geq g(-\Delta/2)$ for all $n\in(-\Delta,0)\cap \mathbb{Z}$. And $f(n)g(n)\leq f(-\Delta/2)g(-\Delta/2)$ for all $n\in(-\Delta,0)\cap \mathbb{Z}$. Also, $f(n)g(n)\geq -\dfrac{(\Delta^2-1)^2}{12}$. Therefore $\dfrac{k^4}{48}\Delta^2(\Delta^2-4)\leq\gamma^2\leq\dfrac{k^4}{36}(\Delta^2-1)^2$. Collision for $\{n,n+\Delta\}=\{-1,1\}$ occur at $\gamma=0$ and all other collision mentioned in Lemma \ref{lem111} occur for $\gamma^2\in\left[\dfrac{k^4}{48}\Delta^2(\Delta^2-4),\dfrac{k^4}{36}(\Delta^2-1)^2\right]$ with $\dfrac{k^4}{48}\Delta^2(\Delta^2-4)>0$. This shows that for each $k>0$, there exist $\gamma_0\neq0$ such that all the collisions stated in Lemma \ref{lem111} occur for $|\g|>|\g_0|$, except $\{n,n+\Delta\}=\{-1,1\}$.

\subsection{Non-periodic perturbations}\label{ss:2} A standard perturbation argument assures that the $L^2(\T)$-spectrum of $\mathcal H_a(\gamma,\tau)$ and $\mathcal H_0(\gamma,\tau)$ will stay close for $|a|$ sufficiently small \cite{Hur2015ModulationalWaves}. Therefore, in order to locate the spectrum of $\mathcal H_a(\gamma,\tau)$, we need to determine the spectrum of $\mathcal H_0(\gamma,\tau)$.
A simple calculation yields that  
\[
\mathcal H_0(\gamma,\tau)) e^{inz} = i\Omega_{n,\gamma,\tau}e^{inz}, \quad n\in\Z,
\]
where
\[
\Omega_{n, \gamma, \tau} = k^{3}(n+\tau)\left(1-(n+\tau)^{2}\right)+\dfrac{3 \gamma^{2}}{k(n+\tau)}.
\]
Therefore, the $L^2(\T)$-spectrum of $\mathcal H_0(\gamma,\tau)$ is given by
\begin{equation}\label{e:spec}
    \operatorname{spec}_{L^2(\mathbb{T})}(\mathcal H_0(\gamma,\tau))=\{i\Omega_{n,\gamma,\tau}; n \in \Z, \tau\in\left(-1/2,1/2\right]\setminus \{0\}\}.
\end{equation}
Since if $\mu \in \operatorname{spec}_{L^2(\mathbb{T})}(\mathcal H_0(\gamma,\tau))$ then $\bar{\mu}\in \operatorname{spec}_{L^2(\mathbb{T})}(\mathcal H_0(\gamma,-\tau))$, therefore, it is enough to consider $\tau\in\left(0,1/2\right]$.
Let $n\neq m\in \Z$, a pair of eigenvalues $i\Omega_{n,\gamma,\tau}$ and $i\Omega_{m,\gamma,\tau}$ of $\mathcal H_0(\gamma,\tau)$ collide for some $\gamma=\gamma_c$ and $\tau\in\left(0,1/2\right]$ when 
\begin{align}\label{e:coll3}
    \Omega_{n,\gamma_c,\tau}=\Omega_{m,\gamma_c,\tau}.
\end{align}
We list all the collisions in the following lemma.
\begin{lemma}\label{lem:coll3}
For a fix $\Delta \in \mathbb{N}$,  eigenvalues $\Omega_{n,\gamma,\tau}$ and $\Omega_{n+\Delta,\gamma,\tau}$ of the operator $\mathcal{H}_0(\gamma,\tau)$ collide for all $n\in[-\Delta,-1]\cap \mathbb{Z}$ along a curve $\gamma=\gamma_c(\tau)$, $\tau\in(0,1/2]$; except $\{n,n+\Delta\}=\{-1,0\}$. All such collisions take place away from the origin in the complex plane except when $\Delta$ is odd and $n=-(\Delta+1)/2$ in which case eigenvalues $\Omega_{n,\gamma,\tau}$ and $\Omega_{-n-1,\gamma,\tau}$ collide at the origin for $\gamma=\gamma_c(1/2)$.
\end{lemma}
\begin{proof}
Without loss of generality, assume that $m>n$ and $m=n+\Delta$ with $\Delta\in \mathbb{N}$. Then from collision condition \eqref{e:coll3}, we obtain
\begin{equation}\label{e:cc1}
    3\gamma^2=-k^4[3(n+\tau)^2(n+\tau+\Delta)^2+(n+\tau)(n+\tau+\Delta)(\Delta^2-1)].
\end{equation}
This implies that collision between $n$ and $n+\Delta$ takes place only if $(n+\tau)(n+\tau+\Delta)<0$, that is, $-\Delta\leq n<0$. In order to check for which $n\in[-\Delta,0)$, there is indeed a collision, assume $n=-t$, $t \in \mathbb{N}$ such that $-t+\tau+\Delta>0$. From collision condition \eqref{e:coll3}, we get 
\begin{equation}\label{e:colll3}
    3\gamma^2\left(\dfrac{1}{t-\tau}+\dfrac{1}{-t+\tau+\Delta}\right)=k^4[(t-\tau)((t-\tau)^2-1)+(-t+\tau+\Delta)((-t+\tau+\Delta)^2-1)].
\end{equation}
There exist such $\gamma$ satisfying \eqref{e:coll3} for all $t$ and $-t+\Delta$, except $\{-t,-t+\Delta\}=\{-1,0\}$. Hence the lemma.
\end{proof}
Note that $\Omega_{n,\gamma,\tau}=0$ at $\gamma^2=-\dfrac{k^4}{3}(n+\tau)^2(1-(n+\tau)^2)$. $\Omega_{n,\gamma_c,\tau}=\Omega_{n+\Delta,\gamma_c,\tau}$=0 for a fixed $\gamma_c$ is possible only for $\Delta=-2n-1$, $\tau=1/2$. Therefore, $\Omega_{n,\gamma_c,\tau}$ and $\Omega_{n+\Delta,\gamma_c,\tau}$ collide at the origin for $n=-(\Delta+1)/2$, for all $n\in[-\Delta,-1]\cap \mathbb{Z}$, $\tau=1/2$ and $\gamma_c^2=\dfrac{k^4(2n+1)^2(4n^2+4n-3)}{48}$; except the pair $\{n,n+\Delta\}=\{-1,0\}$. All other collisions are away from origin. 
From \eqref{e:cc1}, assume  $\gamma^2=-\dfrac{k^4}{3}d(n)h(n)$, where $d(n)=(n+\tau)(n+\tau+\Delta)$ and $h(n)=3(n+\tau)^2+3(n+\tau)\Delta+\Delta^2-1$. 
\begin{equation}\label{e:d}
d(n)=(n+\tau)(n+\tau+\Delta)=(n+\tau)^2+\Delta(n+\tau)=\left(n+\tau+\dfrac{\Delta}{2}\right)^2-\dfrac{\Delta^2}{4}
\end{equation}
\begin{equation}\label{e:h}h(n)=3(n+\tau)^2+3(n+\tau)\Delta+\Delta^2-1=3\left(n+\tau+\dfrac{\Delta}{2}\right)^2+\dfrac{\Delta^2-4}{4}
\end{equation}
From \eqref{e:d} and \eqref{e:h}, for a fixed $\Delta\in \mathbb{N}$, $f(n)\geq-\dfrac{\Delta^2}{4}$ and $g(n)\geq\dfrac{\Delta^2-4}{4}$ for all $n\in\Z$. Collision for $\Delta=2$ occur for $\gamma^2\geq \dfrac{k^4}{4}\tau^3(2-\tau)>0$ and all other collision mentioned in Lemma \ref{lem:coll3} occur for $\gamma^2\geq\dfrac{k^4}{48}\Delta^2(\Delta^2-4)>0$. Also $\gamma^2\leq\dfrac{k^4}{36}(\Delta^2-1)^2$ for all $\Delta\in\mathbb{N}$. Therefore, Collision for $\Delta=2$ occur for $\dfrac{k^4}{4}\tau^3(2-\tau)\leq\gamma^2\leq\dfrac{k^4}{4}$ and all other collisions occur for $\gamma^2\in\left[\dfrac{k^4}{48}\Delta^2(\Delta^2-4),\dfrac{k^4}{36}(\Delta^2-1)^2\right]$ with $\dfrac{k^4}{48}\Delta^2(\Delta^2-4)>0$. This shows that for each $k>0$, there exist $\gamma_0\neq0$ such that all the collisions stated in Lemma \ref{lem:coll3} occur for $|\g|>|\g_0|$.

Since if $\mu \in \operatorname{spec}_{L^2(\mathbb{T})}(\mathcal H_0(\gamma,\tau))$ then $\bar{\mu}\in \operatorname{spec}_{L^2(\mathbb{T})}(\mathcal H_0(\gamma,-\tau))$, there will be collision between conjugate of eigenvalues mentioned in Lemma~\ref{lem:coll3}, for all $\tau\in(-1/2,0)$.
More specifically, collisions for $\{-\Delta,0\}$ occur for all $\tau\in(0,1/2]$, for $\{0,\Delta\}$ occur for all $\tau\in(-1/2,0)$, and the remaining collisions occur for all $\tau\in(-1/2,1/2]$. The perturbation analysis for the collisions mentioned in Lemma~\ref{lem:coll3} will be performed with respect to finite or short wavelength perturbations.

\section{Modulational transverse (in)stabilities}\label{s:4}
Throughout this subsection, we work in the regime $|\gamma|\ll1$, that is, with respect to long-wavelength perturbations. From Lemma~\ref{lem111}, when $\gamma=0$, there is a collision among the eigenvalues $i \Omega_{1,0}$ and $i \Omega_{-1,0}$ at the origin, while all other eigenvalues, on the other hand, remain simple and purely imaginary. Also, in the regime $|\gamma|\ll1$, there is no collision with respect to non-periodic perturbations. Since
\[
\|\mathcal H_a(\gamma) - \mathcal H_0(\g)\|= O(|a|)
\]
as $a \to 0$ uniformly in the operator norm. A standard perturbation argument assures that the spectrum of $\mathcal H_a(\g)$ and $\mathcal H_0(\g)$ will stay close for $|a|$ and $|\gamma|$ small \cite{Hur2015ModulationalWaves}. Therefore, we may write that 
\[
     \operatorname{spec}(\mathcal H_a(\g))=\operatorname{spec}_0(\mathcal H_a(\g)) \cup \operatorname{spec}_1(\mathcal H_a(\g)),
    \] 
for $a$ and $\g$ sufficiently small where $\operatorname{spec}_0(\mathcal H_a(\g))$ contains two eigenvalues bifurcating continuously in $a$ from $i\Omega_{1,0}$ and $i\Omega_{-1,0}$ while $\operatorname{spec}_1(\mathcal H_a(\g))$ consists of infinitely many simple eigenvalues (see \cite{Hur2015ModulationalWaves} and references therein). Further, we investigate if the pair of eigenvalues in $\operatorname{spec}_0(\mathcal H_a(\g))$ bifurcate away from the imaginary axis and contribute to modulational transverse instabilities. 

For $a=0$, $\operatorname{spec}_0(\mathcal H_0(\g))=\{i \Omega_{-1,0},i \Omega_{1,0}\}$ with eigenfunctions $\{e^{-iz},e^{iz}\}$. We choose the real basis $\{\cos z,\sin z\}.$ We calculate expansion of a basis $\{\psi_1,\psi_2\}$ for the eigenspace corresponding to the eigenvalues of $\operatorname{spec}_0(\mathcal H_a(\g))$ in $L^2_0(\mathbb{T})$ by using expansions of $w$ and $c$ in \eqref{e:expptw}, for small $a$ and $\gamma$ as
\begin{align*}
	\psi_{1}(z)  & =\cos z+2 a A_{2} \cos 2 z+3a^2A_3\cos 3z+O(a^{4}),\\
	\psi_{2}(z) & =\sin z+2 a A_{2} \sin 2 z+3a^2A_3\sin 3z+O(a^{4}).
\end{align*}We have the following expression for  $\mathcal{H}_{a}(\gamma)$ after expanding and using $w$ and $c$
\begin{equation}
\begin{aligned}\label{e:expA}
 \mathcal{H}_a(\gamma)  & =\mathcal{H}_0(\gamma)+k a^{2}\left(c_{2}+6 \rho A_{0}-\frac{3}{4} \phi^{2} \right) \partial z+\left(6k\rho a-3\phi^2kA_0a^3-\dfrac{3}{2}\phi^2kA_2a^3\right)\partial_z(\cos z)+\\&ka^2\left(6\rho A_2-\dfrac{3}{4}\phi^2\right)\partial_z(\cos 2z)+\left(6 k \rho a^{3} A_{3} -\frac{3}{2} \phi^{2} k a^{3} A_{2}\right)\partial_{z}(\cos 3 z)+i3\gamma\phi(a\sin z+\\&2a^2A_2\sin 2z+3a^3A_3\sin 3z) \partial_z^{-1}+O(a^4)
  \end{aligned} 
\end{equation}
In order to locate the bifurcating eigenvalues for $|a|$ sufficiently small, we calculate the action of $\mathcal{H}_{a}(\g)$ on the extended eigenspace $\{\psi_1(z), \psi_2(z)\}$ viz.
\begin{align}\label{eq:bmat1}
    \mathcal{T}_{a}(\g) = \left[ \frac{\langle \mathcal{H}_a(\g)\psi_{i}(z),\psi_{j}(z)\rangle}{\langle\psi_{i}(z),\psi_{i}(z)\rangle} \right]_{i,j=1,2}
\text{ and }
    \mathcal{I}_{a} = \left[ \frac{\langle \psi_{i}(z),\psi_{j}(z)\rangle}{\langle\psi_{i}(z),\psi_{i}(z)\rangle} \right]_{i,j=1,2}.
\end{align}
We use expansion of $\mathcal H_a(\gamma)$ in \eqref{e:expA} to find actions of $\mathcal{H}_a(\gamma)$ and identity operator on $\{\psi_1,\psi_2\}$, and arrive at
\begin{align*}
	\mathcal{T}_a(\gamma)&=\begin{pmatrix}
	0 & -\dfrac{3\gamma^2}{k}+3a^2k\left(\dfrac{\phi^2}{4}-\dfrac{\rho^2}{k^2}\right) \\ \dfrac{3\gamma^2}{k} & 0\end{pmatrix}+O(a^2(\gamma+a)),\\
	\end{align*}
	To locate where these two eigenvalues are bifurcating from the origin, we analyze the characteristic equation $	\left|\mathcal{T}_a(\gamma)-\mu \mathcal{I}\right|=0$, where $\mathcal{I}_a$ is $2\times 2$ identity matrix. From which we conclude that
\begin{equation}
    \mu=\pm \dfrac{3|\gamma|}{k}\sqrt{\Lambda} +O(a(\gamma+a))
\end{equation}
where
\begin{equation}\label{e:disc1}
	\Lambda=-\gamma^2+ a^2 k^2 \left(\dfrac{\phi^2}{4}-\dfrac{\rho^{2}}{k^2}\right)+O(a^2(\gamma+a)).
\end{equation}
For $\gamma=a=0$, we get zero as a double eigenvalue, which agrees with our calculation. For $\gamma$ and $a$ sufficiently small, we obtain two eigenvalues which have non-zero real part with opposite sign when
\begin{equation}
    \gamma^2< a^2 k^2 \left(\dfrac{\phi^2}{4}-\dfrac{\rho^{2}}{k^2}\right)+O(a^2(\gamma+a)).
\end{equation}which is possible only for 
\[
k>2\left|\dfrac{\rho}{\phi}\right|.
\]
Hence the Theorem~\ref{t:2}.
\section{High-frequency transverse (in)stabilities}\label{s:5}
As discussed in Subsections \ref{ss:1} and \ref{ss:2}, all the collisions occur for $|\g|>|\g_0|>0$, therefore, here we work in the regime $|\gamma|>|\g_0|>0$, that is, with respect to finite or short wavelength perturbations. Note that, there is no collision for $\Delta=1$ and $2$ among all collisions mentioned in Lemma \ref{lem111}. From Lemma \ref{lem:coll3}, there are collisions for $\Delta=2$ with respect to non-periodic perturbations. For each $\Delta\geq3$, there are collisions for both periodic as well as non-periodic perturbations mentioned in Lemma \ref{lem111} and \ref{lem:coll3}, respectively.\\
\paragraph{\underline{\textbf{(In)stability analysis for $\Delta=2$.}}}
For $\Delta=2$, we have three pairs of colliding eigenvalues $\{\Omega_{-1,\gamma,\tau},\Omega_{1,\gamma,\tau}\}$, $\{\Omega_{0,\gamma,\tau},\Omega_{-2,\gamma,\tau}\}$ and $\{\Omega_{0,\gamma,\tau},\Omega_{2,\gamma,\tau}\}$. We further check if these pairs lead to instability.
Let $i\Omega_{n,\gamma,\tau}$ and $i\Omega_{n+2,\gamma,\tau}$ be such two eigenvalues for some $n\in \Z$. Assume that these eigenvalues collide at $\gamma=\gamma_c$, that is 
\begin{align}
     0 \neq \Omega_{n,\gamma_c,\tau} = \Omega_{n+2,\gamma_c,\tau} = \Omega \hspace{3px} (say).
\end{align}
That is, $i\Omega$ is an eigenvalue of $\mathcal{H}_0(\gamma_c,\tau)$ of multiplicity two with an orthonormal basis of eigenfunctions Let $i \Omega + i \nu_{a,n}$ and $i \Omega + i \nu_{a,n+2}$ be the eigenvalues of $\mathcal{H}_a(\gamma,\tau)$ bifurcating from $i\Omega_{n,\gamma_c,\tau}$ and $i\Omega_{n+2,\gamma_c,\tau}$ respectively, for $|a|$ and $|\g-\g_c|$ small. Let $\{\varphi_{a,n}(z), \varphi_{a,n+2}(z)\}$ be a orthonormal basis for the corresponding eigenspace. We assume the following expansions 
\begin{align}\label{eq:eiggg1}
    \varphi_{a,n}(z) =& e^{inz}+a\varphi_{n,1}+a^2\varphi_{n,2}+O(a^3), \\
    \varphi_{a,n+2}(z) =& e^{i(n+2)z}+a\varphi_{n+2,1}+a^2\varphi_{n+2,2}
      +O(a^3).\label{eq:eiggg2}
\end{align}
We use orthonormality of $\varphi_{a,n,\gamma}$ and $\varphi_{a,n+2,\gamma}$ to find that
\[
\varphi_{n,1} = \varphi_{n,2} = \varphi_{n+2,1} = \varphi_{n+2,2} = 0.
\]Next, we calculate the action of $\mathcal{H}_a(\gamma,\tau)$ on the eigenspace $\{\varphi_{a,n}(z), \varphi_{a,n+2}(z)\}$ for $|\gamma - \gamma_c|$ and $|a|$ small. We arrive at 
\begin{align*}
	\mathcal{T}_a(\gamma,\tau)&=\begin{pmatrix}
		T_{11} &  T_{12} \\ T_{21} & T_{22}\end{pmatrix}+O(a^3(\gamma^2+a^2)),\\
	\end{align*}
where 
\begin{align*}
		T_{11}& =i \Omega+\dfrac{i 3 \varepsilon}{k(n+\tau)}-i(n+\tau) a^{2}k\left(\dfrac{3}{8} \phi^{2} +\dfrac{3}{2} \dfrac{\rho ^{2}}{k^2}\right)  , \\
	T_{12} & =i(n+2+\tau) a^{2}k\left(\dfrac{3}{2} \dfrac{ \rho ^{2}}{k^2}-\dfrac{3}{8} \phi^{2} \right)-\dfrac{i a^{2} 3 \gamma \phi A_{2} }{(n+\tau)},\\
	T_{21}&=i(n+\tau) a^{2}k\left(\dfrac{3}{2} \dfrac{\rho^{2}}{k^2}-\dfrac{3}{8} \phi^{2} \right) +\dfrac{i a^{2} 3 \gamma \phi A_{2}}{(n+2+\gamma)},\\ T_{22}&=i \Omega + \dfrac{i 3 \varepsilon}{k(n+2+\gamma)}-i(n+2+\gamma) a^{2}k\left(\dfrac{3}{8} \phi^{2} +\dfrac{3}{2} \dfrac{\rho ^{2}}{k^2}\right),
\end{align*} 
and $\varepsilon =\gamma^2-\gamma_c^2$, sufficiently small. Further, we obtained the equation $\det(\mathcal{T}_{a}(\gamma,\tau)-(i \Omega + i \nu) \mathcal{I}_{a}) = 0$, where $\mathcal{I}_a$ is $2\times 2$ identity matrix, and concluded the discriminant $\Delta$ as
\begin{align*}
		\Delta=\dfrac{36 \varepsilon^2}{k^2(n+\tau)^{2}(n+2+\tau)^{2}}&-\dfrac{36\gamma_c^2\phi^2a^4A_2^2}{(n+\tau)(n+2+\tau)}+9a^4k^2(n+\tau+1)^2\left(\dfrac{\rho^4}{k^4}+\dfrac{\phi^4}{16}\right)\\&+\dfrac{9a^4\rho^2\phi^2}{2k^2}(1-(n+\tau)(n+2+\tau))+O(a^2|\varepsilon|+ |a|^5)
\end{align*}
Note that all the collisions stated in the Lemma \ref{lem:coll3} for $\Delta=2$ have $(n+\tau)(n+2+\tau)<0$ which implies that for $|\varepsilon|$ and $|a|$ sufficiently small, the leading term in the discriminant is always positive for all $\rho$ and $\phi$. Therefore, we do not get any instability for $\Delta =2$ case for sufficiently small amplitude parameter $a$. \\


\paragraph{\underline{\textbf{(In)stability analysis for $\Delta\geq3$.}}}
For some $n \in \mathbb{Z}^\ast$ and a fixed $\Delta\geq3$, we have
\begin{equation}\label{e:1}
 i \Omega_{n, \gamma_{c},\tau}=i \Omega_{n+\Delta, \gamma_{c},\tau}=i \Omega, \quad \tau\in (-1/2,1/2]
\end{equation}
$\tau=0$ corresponds to the periodic case and $\tau\neq0$ corresponds to the non-periodic case.
\begin{align*}
    \mathcal{H}_a(\gamma)=&\mathcal{H}_0(\gamma)+(\beta_2a^2+\beta_4a^4+\dots)(\partial_z+i\tau)+\alpha_1a(\partial_z+i\tau)(\cos z)+\dots\\&+\alpha_{\Delta}a^{\Delta}(\partial_z+i\tau)(\cos{(\Delta z)})+(i\delta_1a\sin z+\dots+i\delta_{\Delta}a^{\Delta}\sin{(\Delta z)})(\partial_z+i\tau)^{-1}
\end{align*}
To explicitly obtain the values of all unknown coefficients in the expansion of $\mathcal{H}_a(\g,\tau)$, we require coefficients of higher powers of $a$ in the expansion of solution $w$. Calculating higher coefficients is difficult as the coefficients of the solution do not seem to have any apparent symmetry. Therefore, we pursue the instability analysis without calculating the unknown coefficients explicitly. 

Following the same steps as in the previous subsection, we arrive at

\begin{align*}
	\mathcal{T}_a(\gamma,\tau)&=\begin{pmatrix}
		T_{11} & \ T_{12} \\ T_{21} & T_{22}\end{pmatrix}+O(a^{\Delta+1}),\\
	\end{align*}
where 
\begin{align*}
		T_{11}& =i\Omega+\dfrac{i3\varepsilon}{k(n+\tau)}+i(n+\tau)(\beta_2a^2+\beta_4a^4+\dots)  , \\
	T_{12} & =\dfrac{ia^{\Delta}}{2}\left((n+\Delta+\tau)\alpha_{\Delta}-\dfrac{\delta_{\Delta}}{(n+\tau)}\right),\\
	T_{21}&=\dfrac{ia^{\Delta}}{2}\left((n+\tau)\alpha_{\Delta}+\dfrac{\delta_{\Delta}}{n+\Delta+\tau}\right),\\ T_{22}&=i\Omega+\dfrac{i3\varepsilon}{k(n+\Delta+\tau)}+i(n+\Delta+\tau)(\beta_2a^2+\beta_4a^4+\dots),
\end{align*} 
The resulting discriminant of the characteristic equation $\det(\mathcal{T}_{a}(\gamma,\tau)-(i \Omega + i \nu) \mathcal{I}_{a}) = 0$ is
\begin{align*}
    \operatorname{disc}_a(\varepsilon) = \dfrac{9\Delta^2\varepsilon^2}{k^2(n+\tau)^2(n+\Delta+\tau)^2}+\Delta^2\beta_2^2a^4+O(a^2(|\varepsilon|+|a^3|)).
\end{align*}
which is positive for sufficiently small $|\varepsilon|$ and $|a|$ which implies that no eigenvalue of $\mathcal{H}_a(\g)$ is bifurcating from the imaginary axis due to collision. Hence the Theorem \ref{t:3}.
\subsection*{Acknowledgement} 
Bhavna and AKP are supported by the Science and Engineering Research Board (SERB), Department of Science and Technology (DST), Government of India under grant 
SRG/2019/000741. Bhavna is also supported by Junior Research Fellowships (JRF) by University Grant Commission (UGC), Government of India. SS is supported through the institute fellowship from MHRD and  National Institute of Technology, Tiruchirappalli, India.
\bibliographystyle{amsalpha}
\bibliography{ref.bib}

 \end{document}